\documentclass[12pt]{article}
\usepackage{amssymb, amsmath, mathrsfs, amsthm}
\usepackage{enumerate}
\usepackage{graphicx}
\usepackage{url}
\usepackage[margin=1.0in]{geometry}
\usepackage{float, caption, subcaption}
\usepackage[dvipsnames]{xcolor}
\usepackage{optidef}
\newtheorem{theorem}{Theorem}[section]
\newtheorem{lemma}[theorem]{Lemma}
\newtheorem{corollary}[theorem]{Corollary}
\newtheorem{proposition}[theorem]{Proposition}
\newcommand\abs[1]{\left|#1\right|}
\newcommand\ceil[1]{\left\lceil #1\right\rceil}
\newcommand\floor[1]{\left\lfloor #1\right\rfloor}
\DeclareMathOperator{\zf}{Z}
\DeclareMathOperator{\ft}{ft}
\usepackage{tikz}
\usetikzlibrary{positioning, calc}
\tikzset{
    nodeFilledwText/.style={shape=circle,inner sep=2pt,draw=black,fill=lightgray,thick},
    nodeFilledwoText/.style={shape=circle,inner sep=4pt,draw=black,fill=lightgray,thick},
    nodeEmptywText/.style={shape=circle,inner sep=2pt,draw=black,thick},
    nodeEmptywoText/.style={shape=circle,inner sep=4pt,draw=black,thick},
    lineDecorate/.style={black,very thick},
    lab/.style={font=\footnotesize},
    every node/.style={font=\footnotesize}
}
\urldef{\myurl}\url{https://github.com/trcameron/ZFIPModels}

\begin{document}
\title{\bf On the minimal forts of trees}
\author{Thomas R. Cameron\footnote{Department of Mathematics, Penn State Behrend {\tt (trc5475@psu.edu)}}, Kelvin Li\footnote{Department of Mathematics, Penn State Behrend {\tt (kql5755@psu.edu)}}}
\maketitle

\begin{abstract}
In 2018, the concept of a fort in graph theory was introduced as a non-empty subset of vertices satisfying the condition that no vertex outside the set has exactly one neighbor in the set. Since then, forts have played a significant role in characterizing zero forcing sets, modeling the zero forcing number as an integer program, and generating lower bounds for the zero forcing number of Cartesian products. Recent research has focused on the number of minimal forts, defined as those for which no proper subset is a fort. Notably, it has been established that the number of minimal forts in any graph is strictly less than Sperner's bound, a famous bound due to Emanuel Sperner (1928) on the size of a collection of subsets where no subset contains another. Moreover, lower bounds on the number of minimal forts for several families of graphs were established, and it was shown that certain families have an exponential number of minimal forts. In this article, we provide a combinatorial-cut characterization of the minimal forts in trees. Using this characterization, we derive an upper bound on the cardinality of minimal forts and a lower bound on the number of minimal forts in trees. We also characterize the trees that attain this lower bound through a four-part equivalence theorem that provides a connection to other graph parameters, such as star centers, the fort number, and the zero forcing number. 
\\ \\
\textbf{Keywords:} Zero Forcing; Minimal Forts; Star Centers; Trees
\\
\textbf{AMS Subject Classification:} 05C05, 05C15, 05C30, 05C57
\end{abstract}
\thispagestyle{empty}

\section{Introduction}\label{sec:intro}
Zero forcing is a binary coloring game on a graph in which a set of filled vertices can force non-filled vertices to become filled according to a color change rule. 
In 2007, zero forcing was independently developed by Burgarth and Giovannetti in the study of controllability of quantum systems~\cite{Burgarth2007}, where it was referred to as graph infection.
In 2008, the zero forcing number was shown to be an upper bound on the maximum nullity of a real symmetric matrix associated with a graph~\cite{AIM2008}.
Since then, zero forcing and its applications have been extensively studied; see~\cite{Hogben2022} and the references therein.

The decision problem for the zero forcing number is NP-complete~\cite{Aazami2008}. 
Consequently, there has been significant interest in developing optimization models for studying this parameter.
Several integer programming models have been proposed for computing the zero forcing number and other related graph parameters such as the propagation time, throttling number, and fort number~\cite{Agra2019,Brimkov2019,Brimkov2021,Cameron2025}.
One such model is the fort cover model, which models zero forcing as a set cover problem in which the sets correspond to forts of a graph.

The forts of a graph were introduced in~\cite{Fast2018} and are closely related to failed zero forcing sets~\cite{Fetcie2015,Swanson2023}. 
In particular, the complement of a stalled failed zero forcing set is a fort; see the closure-complement method described in~\cite{Brimkov2019}. 
In the literature, forts are also referred to as zero-blocking sets; see, for example,~\cite{Beaudouin2020,Lin2025_comp,Lin2025_grid}. 
These sets capture structural obstructions to the propagation of the zero forcing process and play an important role in both theoretical and computational approaches to the zero forcing number. 
Recently, in~\cite{Lin2025_comp}, a linear time algorithm was developed for computing a minimum fort in a tree.
This result is notable since the decision problem for the maximum failed zero forcing number is NP-complete~\cite{Shitov2017}.

A fort is minimal if no proper subset is also a fort.
Thus, the minimal forts represent the minimal structural obstructions to the propagation of the zero forcing process and form the natural building blocks for the fort cover formulation of the zero forcing problem. 
In recent work, Becker et al.~\cite{Becker2025} studied the number of minimal forts in graphs.
They proved that the number of minimal forts in any graph is strictly less than Sperner's bound, a classical bound due to Emanuel Sperner (1928) on the size of a family of subsets in which no subset contains another.
They also derived explicit formulas and lower bounds for the number of minimal forts for several graph families and showed that certain families contain an exponential number of minimal forts.

More recently, Grood et al.~\cite{Grood2024} and Jacob et al.~\cite{Jacob2025} introduced the concept of star centers.
In~\cite{Jacob2025}, it was shown that every star center is fort irrelevant, meaning that no star center is contained in any minimal fort. 
For trees, it was further shown that a vertex is a star center if and only if it is fort irrelevant. 
Moreover, fort irrelevance was shown to coincide with zero forcing irrelevance.

In this article, we investigate the structure and enumeration of minimal forts in trees. 
In Section~\ref{sec:mf_tree_prop}, we prove a combinatorial-cut characterization of the minimal forts of a tree. 
This characterization is particularly useful since the conditions are generally much easier to check than the condition that no proper subset is a fort. 
We use this characterization to provide an upper bound on the cardinality of a minimal fort and a lower bound on the number of minimal forts of a tree. 

The lower bound on the number of minimal forts of a tree requires multiple steps. 
First, we establish lower bounds for several classes of trees.
In particular, we show that paths of order $n\geq 6$ (Section~\ref{subsec:mf_path_enumerate}), spiders of order $n\geq 4$ (Section~\ref{subsec:mf_spider_enumerate}), and trees of order $n\geq 6$ with no star centers (Section~\ref{subsec:mf_no_star_centers_enumerate}) all have at least $n/2$ minimal forts. 
Then, in Section~\ref{subsec:mf_with_star_centers_enumerate} we establish an inequality relation the number of minimal forts and the number of star centers. 
We use this relationship to prove that every tree of order $n\geq 1$ has at least $n/3$ minimal forts.
Finally, in Section~\ref{sec:equiv_thm}, we characterize the trees that attain this lower bound through a four-part equivalence theorem involving the number of minimal forts, the number of star centers, the fort number, and the zero forcing number.

\section{Preliminaries}\label{sec:prelim}
Throughout this article, we let $G=(V,E)$ denote a finite simple unweighted graph, where $V$ is the vertex set, $E$ is the edge set, and $uv\in E$ if and only if $u\neq v$ and there is an edge between vertices $u$ and $v$.
If the context is not clear, we use $V(G)$ and $E(G)$ to specify the vertex set and edge set of $G$, respectively. 
The \emph{order} of $G$ is denoted by $n=\abs{V}$ and the \emph{size} of $G$ is denoted by $m=\abs{E}$; when $m=0$, we refer to $G$ as the \emph{empty graph}.

Let $G=(V,E)$ and $u\in V$.
We define the \emph{neighborhood} of $u$ as $N(u) = \left\{v\in V\colon uv\in E\right\}$.
We refer to every $v\in N(u)$ as a \emph{neighbor} of $u$, and we say that $u$ and $v$ are \emph{adjacent}. 
If the context is not clear, we use $N_{G}(u)$ to denote the neighborhood of $u$ in the graph $G$.

If $u,v\in V$ satisfy $N(u)=N(v)$, then we say that $u$ and $v$ are \emph{twins}. 
The degree of $u$ is denoted by $\deg(u)=\abs{N(u)}$; when $\deg(u)=0$, we refer to $u$ as an \emph{isolated vertex}, when $\deg(u)=1$, we refer to $u$ as a \emph{pendant vertex} or a \emph{leaf}, and when $\deg(u)\geq 3$, we refer to $u$ as a \emph{junction vertex}.
We use $\ell(G)$ to denote the number of pendant vertices in a graph $G$. 

The graph $G'=(V',E')$ is a \emph{subgraph} of $G$ if $V'\subseteq V$ and $E'\subseteq E$. 
Given $W\subseteq V$, the \emph{induced subgraph}, denoted $G[W]$, is the subgraph of $G$ made up of the vertices in $W$ and the edge $uv\in E$ such that $u,v\in W$. 

Given vertices $u,v\in V$, we say that $u$ is \emph{connected to} $v$ if there exists a list of distinct vertices $(w_{0},w_{1},\ldots,w_{l})$ such that $u=w_{0}$, $v=w_{l}$, $w_{i}w_{i+1}\in E$ for all $i\in\{0,1,\ldots,l-1\}$, and $w_{i}\neq w_{j}$ for all $i\neq j$. 
Such a list of vertices is called a \emph{path}; in particular, we reference it as a $(u,v)$-path. 
The connected to relation is an equivalence relation, which partitions the vertex set $V$ into disjoint subsets: $W_{1},W_{2},\ldots,W_{k}$, where for all $i=1,\ldots,k$ and for all $u,v\in W_{i}$, $u$ is connected to $v$. 

For each $i=1,\ldots,k$, the induced subgraph $G[W_{i}]$ is called a \emph{connected component} of $G$.
The graph $G$ is \emph{connected} if it has exactly one connected component. 
Given $u\in V$, the graph $G-u$ is the induced subgraph $G[V\setminus\{u\}]$.
Also, given $uv\in E$, the graph $G-uv$ is the subgraph of $G$ obtained by deleting the edge $uv$. 
If $G-u$ has more connected components than $G$, then $u$ is a \emph{cut vertex} of $G$.
Similarly, if $G-uv$ has more connected components than $G$, then $uv$ is a \emph{cut edge} of $G$. 

Another graph operation that plays an important role in this article is the corona product. 
The \emph{corona} of $G=(V,E)$ with $G'=(V',E')$, denoted $G\circ G'$, is the graph obtained from the disjoint union of $G$ and $\abs{V}$ copies of $G'$ by joining each vertex $u\in V$ with the $u-$copy of $G'$. 

A \emph{cycle} is a list of vertices $(w_{0},w_{1},\ldots,w_{l})$ such that $w_{i}w_{i+1}\in E$ for all $i\in\{0,1,\ldots,l-1\}$, $w_{i}\neq w_{j}$ for all $i\neq j$, and $w_{0}w_{l}\in E$. 
A graph that contains no cycles is called \emph{acyclic}. 
A tree $T$ is a graph that is connected and \emph{acyclic}. 
The following theorem states well-known equivalent characterizations of a tree, we make use of these conditions throughout the article. 
\begin{theorem}[\cite{Diestel2016}]\label{thm:tree_char}
The following conditions are equivalent for a graph $G=(V,E)$. 
\begin{enumerate}[(a)]
\item $G=(V,E)$ is a tree.
\item For any $u,v\in V$, there is a unique $(u,v)$-path.
\item Every edge $uv\in E$ is a cut edge.
\item $G$ is connected and $\abs{E}=\abs{V}-1$. 
\end{enumerate}
\end{theorem}

\emph{Zero forcing} is a binary coloring game on a graph, where vertices are either filled or non-filled. 
In this article, we denote filled vertices by the color gray and non-filled vertices by the color white. 
An initial set of gray vertices can force white vertices to become gray following a color change rule. 
While there are many color change rules, see~\cite[Chapter 9]{Hogben2022}, we will use the \emph{standard rule} which states that a gray vertex $u$ can force a white vertex $v$ if $v$ is the only white neighbor of $u$.
Since the vertex set is finite, there comes a point in which no more forcings are possible.
If at this point all vertices are gray, then we say that the initial set of gray vertices is a \emph{zero forcing set} of $G$; otherwise, we refer to the initial set of gray vertices as a \emph{failed zero forcing set} of $G$.
The \emph{zero forcing number} of $G$, denoted $\zf(G)$, is the minimum cardinality of a zero forcing set of $G$.

Let $G=(V,E)$ be a graph. 
A fort of $G$ is a non-empty subset of vertices $F\subseteq V$ such that every vertex $v\in V\setminus{F}$ satisfies $\abs{N(v)\cap F}\neq 1$. 
A fort $F$ is a minimal fort of $G$ if every proper subset of $F$ is not a fort of $G$. 
We let $\mathcal{F}_{G}$ denote the set of all minimal forts of $G$. 
The fort number, denoted $\ft(G)$, is the maximum of a collection of disjoint sets from $\mathcal{F}_{G}$. 

Forts are also referenced as zero-blocking sets in the literature, see for example~\cite{Beaudouin2020,Lin2025_comp,Lin2025_grid}. 
It is clear from the definition of a fort that a set of vertices is a zero forcing set if and only if the set intersects every minimal fort of the graph. 
This observation has been used to generate bounds on the zero forcing number and model the zero forcing number as a set cover problem~\cite{Brimkov2019,Cameron2023,Fast2018}.
For reference, we state this observation in the following theorem. 
\begin{theorem}[Theorem 8 in~\cite{Brimkov2019}]\label{thm:fort_cover}
Let $G=(V,E)$ be a graph.
Then, $S\subseteq V$ is a zero forcing set of $G$ if and only if $S$ intersects every minimal fort in $\mathcal{F}_{G}$.
\end{theorem}

A star center of $G$ is defined through the star removal process; note that star centers were introduced as $B-$vertices in~\cite{Grood2024}.
Let $G_{0}=G$ and define $S_{0}$ to be the vertices that have two or more pendant neighbors in $G_{0}$.
Then, define $G_{1}$ to be the graph formed by removing all vertices in $S_{0}$ from $G_{0}$ as well as their pendant neighbors. 
Continuing for $i\geq 1$, let $S_{i}$ be the vertices that have two or more pendant neighbors in $G_{i}$.
Then, define $G_{i+1}$ to be the graph formed by removing all vertices in $S_{i}$ from $G_{i}$ as well as their pendant neighbors. 
Since the graph $G$ is finite, there exists a $t\geq 0$ such that $S_{t}$ is empty. 
We let $\mathcal{S}_{G}$ denote the \emph{star centers} of $G$, which is defined by the union 
\[
\mathcal{S}_{G} = S_{0}\cup S_{1} \cup \cdots \cup S_{t}
\]
Star centers are closely related to fort irrelevant vertices, which are defined as vertices that are not members of any minimal fort.
Similarly, zero forcing irrelevant vertices are those that are not members of any minimal zero forcing set. 
In~\cite{Jacob2025}, the authors prove the following useful results regarding star centers and irrelevant vertices. 
\begin{theorem}[Theorem 2.4 in \cite{Jacob2025}]\label{thm:irrelevant}
Let $G=(V,E)$ be a graph and let $v\in V$.
Then, $v$ is fort irrelevant if and only if $v$ is zero forcing irrelevant. 
\end{theorem}
\begin{theorem}[Corollary 2.5 in \cite{Jacob2025}]\label{thm:tree_star_centers}
Let $T=(V,E)$ be a tree and let $v\in V$.
Then, $v$ is fort irrelevant if and only if $v$ is a star center of $T$. 
\end{theorem}
\begin{theorem}[Corollary 2.6 in \cite{Jacob2025}]\label{thm:star_centers}
Let $G=(V,E)$ be a graph and let $v\in V$.
If $v$ is a star center of $G$, then $v$ is fort irrelevant. 
\end{theorem}

\section{Properties of Minimal Forts in Trees}\label{sec:mf_tree_prop}
In this section, we present a combinatorial-cut characterization of the minimal forts of trees. 
This characterization is used to give a bound on the size of minimal forts of trees and to provide constructions that are useful when enumerating the minimal forts of a tree.  
We begin with the following lemma which states that no vertex in a fort has more than one neighbor in the fort. 
\begin{lemma}\label{lem:fort_nbhd_in}
Let $T$ be a tree and let $F$ be a minimal fort of $T$. 
Let $u\in F$.
Then, $u$ has at most one neighbor in $F$. 
\end{lemma}
\begin{proof}
For the sake of contradiction, suppose that there exists a $u\in F$ such that 
\[
\abs{N(u)\cap F} \geq 2.
\]
Let $v_{1},\ldots,v_{k}$, where $k\geq 0$, denote the neighbors of $u$ that are not in $F$. 
Since $T$ is a tree, each edge $uv_{i}$ is a cut edge. 
For $1\leq i\leq k$, let $C_{i}$ denote the component of $T-uv_{i}$ that contains the vertex $v_{i}$. 
Define 
\[
F' = F\setminus\left(\left(F\cap V(C_{1})\right)\cup\cdots\cup \left(F\cap V(C_{k})\right)\cup\{u\}\right).
\]
We claim that $F'$ is a fort of $T$.
To this end, let $v\in V(T)\setminus{F'}$ and consider two cases $v\in F$ or $v\notin F$. 

Suppose that $v\in F$.
Then, $v=u$ or $v\in F\cap V(C_{i})$ for some $i\in\{1,\ldots,k\}$. 
If $v=u$, then 
\[
\abs{N(v)\cap F'} = \abs{N(u)\cap F} \geq 2.
\]
If $v\in F\cap V(C_{i})$, then $\abs{N_{T}(v)\cap F'}=0$.

Suppose that $v\notin F$.
If $v\in V(C_{i})$ for some $i\in\{1,\ldots,k\}$, then $\abs{N_{T}(v)\cap F'}=0$.
Otherwise,
\[
\abs{N_{T}(v)\cap F'} = \abs{N_{T}(v)\cap F}\neq 1.
\]

Since every $v\in V(T)\setminus{F'}$ satisfies $\abs{N_{T}(v)\cap F'}\neq 1$, it follows that $F'$ is a fort of $T$.
Furthermore, $F'\subsetneq F$, which contradicts the minimality of $F$.
\end{proof}

Next, we prove that no vertex outside of a fort has more than two neighbors in the fort. 
\begin{lemma}\label{lem:fort_nbhd_out}
Let $T$ be a tree and let $F$ be a minimal fort of $T$. 
Let $v\in V(T)\setminus{F}$.
Then, $v$ has zero neighbors in $F$ or exactly $2$ neighbors in $F$.
\end{lemma}
\begin{proof}
For the sake of contradiction, suppose there exists a $v\in V(T)\setminus{F}$ such that 
\[
\abs{N_{T}(v)\cap F}\geq 3.
\]
Let $u_{1},\ldots,u_{k}$, where $k\geq 3$, denote the neighbors of $v$ in $F$. 
Since $T$ is a tree, each edge $vu_{i}$ is a cut edge.
For $1\leq i\leq k$, let $C_{i}$ denote the component of $T-vu_{i}$ that contains the vertex $u_{i}$. 
Define
\[
F' = F\setminus\left(\left(F\cap V(C_{3})\right)\cup\cdots\cup\left(F\cap V(C_{k})\right)\right)
\]
We claim that $F'$ is a fort of $T$.
To this end, let $w\in V(T)\setminus{F'}$ and consider three cases $w=v$, $w\in V(C_{i})$ for some $i\in\{1,2\}$, or $w\in V(C_{i})$ for some $i\in\{3,\ldots,k\}$.

If $w=v$, then $\abs{N_{T}(w)\cap F'} = 2$.
If $w\in V(C_{i})$ for some $i\in\{1,2\}$, then 
\[
\abs{N_{T}(w)\cap F'} = \abs{N_{T}(w)\cap F} \neq 1.
\]
If $w\in V(C_{i})$ for some $i\in\{3,\ldots,k\}$, then $\abs{N_{T}(w)\cap F'}=0$.

Since every $w\in V(T)\setminus{F'}$ satisfies $\abs{N_{T}(w)\cap F'}\neq 1$, it follows that $F'$ is a fort of $T$.
Furthermore, $F'\subsetneq F$, which contradicts the minimality of $F$.
\end{proof}

The conditions in Lemma~\ref{lem:fort_nbhd_in} and Lemma~\ref{lem:fort_nbhd_out} are necessary for minimal forts of trees.
When combined with a cut condition, we obtain a combinatorial-cut characterization of the minimal forts of trees.
\begin{theorem}\label{thm:mf_comb_cut_char}
Let $T=(V,E)$ be a tree and let $F\subseteq V$ be non-empty.
Then, $F$ is a minimal fort of $T$ if and only if the following conditions are satisfied:
\begin{enumerate}[(a)]
\item For every $u\in F$, $\abs{N(u)\cap F}\leq 1$,
\item For every $v\notin F$, $\abs{N(v)\cap F}\in\{0,2\}$.
\item For every $x,y\notin F$ where $xy\in E$, $F$ lies entirely in one component of $T-xy$. 
\end{enumerate}
\end{theorem}
\begin{proof}
Suppose that $F$ is a minimal fort of $T$.
Then, conditions (a) and (b) hold by Lemma~\ref{lem:fort_nbhd_in} and Lemma~\ref{lem:fort_nbhd_out}, respectively. 
Suppose there exist vertices $x,y\notin F$ such that $xy\in E$.
Let $C_{x}$ and $C_{y}$ denote the connected components of $T-xy$ that contain $x$ and $y$, respectively. 
For the sake of contradiction, suppose that $F$ has vertices $u\in V(C_{x})$ and $v\in V(C_{y})$.
Let $W=\{w_{0},w_{1},\ldots,w_{l}\}$ denote the vertices of the unique $(u,v)-$path in $T$.
Then, Theorem 16 of~\cite{Becker2025} states that $F\cap W$ is a minimal fort of $T[W]$.
Since $x$ and $y$ lie on the $(u,b)-$path and both are outside of $F$, the induced path $T[W]$ contains two adjacent vertices outside of $F\cap W$, which is impossible for a fort on a path. 
Therefore, condition (c) must hold. 

Conversely, suppose that $F\subseteq V$ is non-empty and satisfies conditions (a)--(c).
Condition (b) clearly implies that $F$ is a fort of $T$. 
For the sake of contradiction, suppose that $F$ is not minimal. 
Then, there exists a minimal fort $F'\subsetneq F$.
Let $u\in F'$ and $w\in F\setminus{F'}$.
Also, let $W=\{w_{0},w_{1},\ldots,w_{l}\}$ denote the vertices in the unique $(u,w)-$path in $T$, where $u=w_{0}$ and $w=w_{l}$.
Without loss of generality, we can assume that $w_{1},\ldots,w_{l}\in V\setminus{F'}$; otherwise, we can just shorten the path by selecting a different $u\in F'$.
Since $F'$ is a minimal fort and $w_{1}$ is adjacent to $w_{0}\in F'$, Lemma~\ref{lem:fort_nbhd_out} implies that $w_{1}$ is adjacent to $2$ vertices in $F'$.
If $w_{1}\in F$, then this implies that there is a $w_{1}\in F$ such that $\abs{N(w_{1})\cap F}\geq 2$, which contradicts (a). 
Suppose that $w_{1}\notin F$.
If $w_{2}\notin F$, then $F$ would have vertices in both connected components of $T-w_{1}w_{2}$, which contradicts condition (c).
Therefore, $w_{2}\in F$.
However, this implies that there is a vertex $w_{1}\notin F$ such that $\abs{N(w_{1})\cap F}\geq 3$, which contradicts (b).
\end{proof}

The combinatorial-cut characterization of the minimal forts of trees is particularly useful since the conditions in Theorem~\ref{thm:mf_comb_cut_char} are generally much easier to check than the condition that no proper subset is a fort. 
Moreover, these conditions can be used to derive an upper bound on the cardinality of a minimal fort of a tree, as shown in the following theorem.
\begin{theorem}\label{thm:mf_card_bound}
Let $T$ be a tree and let $F$ be a minimal fort of $T$.
Then, 
\[
\abs{F}\leq\floor{\frac{2n+2}{3}}.
\]
\end{theorem}
\begin{proof}
Let $C_{1},\ldots,C_{k}$ denote the connected components of $T[F]$ and let 
\[
B = \left\{x\in V(T)\setminus{F}\colon N_{T}(x)\cap F\neq\emptyset\right\}
\]
denote the set of boundary vertices outside of $F$. 
Define $J$ to be the bipartite graph with vertex set
\[
V(J) = \left\{C_{1},\ldots,C_{k}\right\}\cup B
\]
and edge set
\[
E(J) = \left\{\{C_{i},x\}\colon N_{T}(x)\cap V(C_{i})\neq\emptyset\right\}.
\]
Note that cycles in $J$ correspond to cycles in $T$.
Therefore, since $T$ is a tree, $J$ is acyclic. 
Furthermore, disconnected components in $J$ correspond to boundary vertices $b,b'\in B$, where the $(b,b')-$path in $T$ only contains vertices not in $F$, which contradicts condition (c) of Theorem~\ref{thm:mf_comb_cut_char}.
Since $F$ is a minimal fort, it follows that $J$ is connected. 

Since $J$ is a tree, it follows that $\abs{E(J)} = k + \abs{B}-1$.
Furthermore, by condition (b) of Theorem~\ref{thm:mf_comb_cut_char}, every $x\in B$ has exactly two neighbors in $F$ and these neighbors must come from distinct components of $T[F]$ in order to avoid cycles.
Hence, $\abs{E(J)} = 2\abs{B}$.
Therefore, we have $k = \abs{B}+1$. 
Now, by condition (a) of Theorem~\ref{thm:mf_comb_cut_char}, each component of $T[F]$ contains at most $2$ vertices. 
Hence, 
\begin{align*}
\abs{F} \leq 2k &= 2\left(\abs{B}+1\right) \\
&\leq 2\left(\abs{V(T)\setminus{F}} + 1\right) \\
&= 2\left(n-\abs{F}+1\right),
\end{align*}
which implies that $3\abs{F} \leq 2n+2$.
\end{proof}

In Figure~\ref{fig:extremal_fort_and_matching} we display a minimal fort of the path graph $P_{8}$ and the corresponding induced subgraph.
Note that the minimal fort has maximum cardinality as described by Theorem~\ref{thm:mf_card_bound}.
Moreover, the corresponding induced subgraph is a matching as implied by Theorem~\ref{thm:mf_comb_cut_char}, where every vertex outside of the fort is a witness to $0$ or $2$ vertices in the fort and no pair of vertices outside of the fort have a cut edge between them that split the fort across multiple components. 
\begin{figure}[ht]
\centering
\begin{tikzpicture}[scale=0.65]
\begin{scope}[xshift=0em]
    \node (a1) at (0,0)   [nodeEmptywoText] {};
    \node (b1) at (1,0)   [nodeEmptywoText] {};
    \node (a2) at (3,0)   [nodeEmptywoText] {};
    \node (b2) at (4,0)   [nodeEmptywoText] {};
    \node (a3) at (6,0)   [nodeEmptywoText] {};
    \node (b3) at (7,0)   [nodeEmptywoText] {};
    \node (x)  at (2,1.2) [nodeFilledwoText] {};
    \node (y)  at (5,1.2) [nodeFilledwoText] {};
    \draw[lineDecorate] (a1)--(b1);
    \draw[lineDecorate] (a2)--(b2);
    \draw[lineDecorate] (a3)--(b3);
    \draw[lineDecorate] (x)--(a1);
    \draw[lineDecorate] (x)--(a2);
    \draw[lineDecorate] (y)--(b2);
    \draw[lineDecorate] (y)--(a3);
    \node[lab] at (3.5,-0.8) {Tree $T$ with fort $F$ vertices in white};
\end{scope}
\begin{scope}[xshift=25em]
    \node (A1) at (0,0)   [nodeEmptywoText] {};
    \node (B1) at (1.2,0) [nodeEmptywoText] {};
    \node (A2) at (3,0)   [nodeEmptywoText] {};
    \node (B2) at (4.2,0) [nodeEmptywoText] {};
    \node (A3) at (6,0)   [nodeEmptywoText] {};
    \node (B3) at (7.2,0) [nodeEmptywoText] {};
    \draw[lineDecorate] (A1)--(B1);
    \draw[lineDecorate] (A2)--(B2);
    \draw[lineDecorate] (A3)--(B3);
    \node[lab] at (4.0,-1.0) {Induced subgraph $T[F]$};
\end{scope}
\end{tikzpicture}
\caption{A minimal fort $F$ attaining the bound in Theorem~\ref{thm:mf_card_bound} (left), and the induced subgraph $T[F]$ (right), which is a matching as implied by Theorem~\ref{thm:mf_comb_cut_char}.}
\label{fig:extremal_fort_and_matching}
\end{figure}
\section{Enumerating the Minimal Forts of Trees}\label{sec:mf_tree_enumerate}
In this section, we provide lower bounds on the number of minimal forts of a tree.

\subsection{Enumerating the Minimal Forts of Paths}\label{subsec:mf_path_enumerate}
The following propositions provide lower bounds on the number of minimal forts of a path graph.
Each proposition makes use of the following recursive formula~\cite[Corollary 12]{Becker2025}:
\begin{equation}\label{eq:mf_path_recursive_formula}
\abs{\mathcal{F}_{P_{n}}} = \abs{\mathcal{F}_{P_{n-2}}} + \abs{\mathcal{F}_{P_{n-3}}},
\end{equation}
for $n\geq 4$, where $\abs{\mathcal{F}_{P_{1}}}=1$, $\abs{\mathcal{F}_{P_{2}}}=1$, and $\abs{\mathcal{F}_{P_{3}}}=1$.

\begin{proposition}\label{prop:path_lower_bound}
Let $P_{n}$ be a path graph of order $n\geq 1$.
If $n=1$, then $P_{n}$ has exactly one minimal fort. 
For $n\geq 2$, we have
\[
\abs{\mathcal{F}_{P_{n}}} \geq \floor{\frac{n}{2}},
\]
where $\floor{\cdot}$ denotes the floor function.
\end{proposition}
\begin{proof}
The lower bound holds for paths of order $n=1,2,3$, each of which has exactly one minimal fort.
Applying~\eqref{eq:mf_path_recursive_formula}, it follows that the lower bound holds for all $n\geq 1$.
\end{proof}
\begin{proposition}\label{prop:path_sharper_lower_bound}
Let $P_{n}$ be a path graph of order $n\geq 6$.
Then, we have
\[
\abs{\mathcal{F}_{P_{n}}} \geq \ceil{\frac{n}{2}},
\]
where $\ceil{\cdot}$ denotes the ceiling function.
\end{proposition}
\begin{proof}
The lower bound holds for paths of order $n=6,7,8$, which have $3,4,5$ minimal forts, respectively. 
Applying~\eqref{eq:mf_path_recursive_formula}, it follows that the lower bound holds for all $n\geq 6$. 
\end{proof}

\subsection{Enumerating the Minimal Forts of Spiders}\label{subsec:mf_spider_enumerate}
Next, we prove that the $\ceil{\frac{n}{2}}$ lower bound holds for spider graphs.
We make use of~\cite[Theorem 17]{Becker2025} which provides a formula for the number of minimal forts of a spider graph $S_{l_{1},\ldots,l_{k}}$, with junction vertex $v$, pendant vertices $u_{1},\ldots,u_{k}$, where $k\geq 3$, and for each $i\in\{1,\ldots,k\}$ the path from $v$ to $u_{i}$ has length $l_{i}\geq 1$.
In particular, the number of minimal forts of the spider graph satisfies 
\begin{equation}\label{eq:mf_spider}
\abs{\mathcal{F}_{S_{l_{1},\ldots,l_{k}}}} = \prod_{i=1}^{k}\abs{\mathcal{F}_{P_{l_{i}-1}}} + \sum_{i=1}^{k}\abs{\mathcal{F}_{P_{l_{i}-2}}}\prod_{j\neq i}\abs{\mathcal{F}_{P_{l_{j}-1}}}+\sum_{1\leq i<j\leq k}\abs{\mathcal{F}_{P_{l_{i}}}}\abs{\mathcal{F}_{P_{l_{j}}}},
\end{equation}
where $\abs{\mathcal{F}_{P_{-1}}}=1$ and $\abs{\mathcal{F}_{P_{0}}}=0$.
Note that the spider graph has $\sum_{1\leq i<j\leq k}\abs{\mathcal{F}_{P_{l_{i}}}}\abs{\mathcal{F}_{P_{l_{j}}}}$ minimal forts that do not contain $v$, $\prod_{i=1}^{k}\abs{\mathcal{F}_{P_{l_{i}-1}}}$ minimal forts that include $v$ and no neighbor of $v$, and $\sum_{i=1}^{k}\abs{\mathcal{F}_{P_{l_{i}-2}}}\prod_{j\neq i}\abs{\mathcal{F}_{P_{l_{j}-1}}}$ minimal forts that contain $v$ and exactly one neighbor of $v$.
We begin with the following lemma which counts the change in the number of minimal forts when a vertex is appended to a leg of a spider.
Since the labeling of the legs is arbitrary, we may assume that the vertex is appended to the first leg of the spider. 
\begin{lemma}\label{lem:increase_leg_length1}
Let $S_{l_{1},\ldots,l_{k}}$ be a spider graph and let $S_{l_{1}+1,\ldots,l_{k}}$ be constructed by appending a vertex to the first leg of $S_{l_{1},\ldots,l_{k}}$. 
Then, 
\[
\abs{\mathcal{F}_{S_{l_{1}+1,\ldots,l_{k}}}} - \abs{\mathcal{F}_{S_{l_{1},\ldots,l_{k}}}} = \Delta_{l_{1}-2}^{l_{1}}\prod_{i=2}^{k}\abs{\mathcal{F}_{P_{l_{i}-1}}} + \Delta_{l_{1}-1}^{l_{1}}\sum_{i=2}^{k}\abs{\mathcal{F}_{P_{l_{i}-2}}}\prod_{\substack{j\neq i \\ j\neq 1}}\abs{\mathcal{F}_{P_{l_{j}-1}}} + \Delta_{l_{1}}^{l_{1}+1}\sum_{i=2}^{k}\abs{\mathcal{F}_{P_{l_{i}}}},
\]
where $\Delta_{a}^{b}=\abs{\mathcal{F}_{P_{b}}} - \abs{\mathcal{F}_{P_{a}}}$, $\abs{\mathcal{F}_{P_{-1}}}=1$, and $\abs{\mathcal{F}_{P_{0}}}=0$.
\end{lemma}
\begin{proof}
By factoring out the $i=1$ term of each expression in \eqref{eq:mf_spider}, we have 
\begin{align*}
\abs{\mathcal{F}_{S_{l_{1},\ldots,l_{k}}}} = \abs{\mathcal{F}_{P_{l_{1}-1}}}\prod_{i=2}^{k}\abs{\mathcal{F}_{P_{l_{i}-1}}} &+ \abs{\mathcal{F}_{P_{l_{1}-2}}}\prod_{j\neq 1}\abs{\mathcal{F}_{P_{l_{j}-1}}} + \abs{\mathcal{F}_{P_{l_{1}-1}}}\sum_{i=2}^{k}\abs{\mathcal{F}_{P_{l_{i}-2}}}\prod_{j\neq i,1}\abs{\mathcal{F}_{P_{l_{j}-1}}} \\
&+ \abs{\mathcal{F}_{P_{l_{1}}}}\sum_{j=2}^{k}\abs{\mathcal{F}_{P_{l_{j}}}} + \sum_{2\leq i<j\leq k}\abs{\mathcal{F}_{P_{l_{i}}}}\abs{\mathcal{F}_{P_{l_{j}}}}. 
\end{align*}
Applying \eqref{eq:mf_spider} to $S_{l_{1}+1,\ldots,l_{k}}$ and factoring out the $i=1$ term from each expression, gives
\begin{align*}
\abs{\mathcal{F}_{S_{l_{1}+1,\ldots,l_{k}}}} = \abs{\mathcal{F}_{P_{l_{1}}}}\prod_{i=2}^{k}\abs{\mathcal{F}_{P_{l_{i}-1}}} &+ \abs{\mathcal{F}_{P_{l_{1}-1}}}\prod_{j\neq 1}\abs{\mathcal{F}_{P_{l_{j}-1}}} + \abs{\mathcal{F}_{P_{l_{1}}}}\sum_{i=2}^{k}\abs{\mathcal{F}_{P_{l_{i}-2}}}\prod_{j\neq i,1}\abs{\mathcal{F}_{P_{l_{j}-1}}} \\
&+ \abs{\mathcal{F}_{P_{l_{1}+1}}}\sum_{j=2}^{k}\abs{\mathcal{F}_{P_{l_{j}}}} + \sum_{2\leq i<j\leq k}\abs{\mathcal{F}_{P_{l_{i}}}}\abs{\mathcal{F}_{P_{l_{j}}}}. 
\end{align*}
By taking the difference and combining like terms, we have
\[
\abs{\mathcal{F}_{S_{l_{1}+1,\ldots,l_{k}}}} - \abs{\mathcal{F}_{S_{l_{1},\ldots,l_{k}}}} = \Delta_{l_{1}-2}^{l_{1}}\prod_{i=2}^{k}\abs{\mathcal{F}_{P_{l_{i}-1}}} + \Delta_{l_{1}-1}^{l_{1}}\sum_{i=2}^{k}\abs{\mathcal{F}_{P_{l_{i}-2}}}\prod_{\substack{j\neq i \\ j\neq 1}}\abs{\mathcal{F}_{P_{l_{j}-1}}} + \Delta_{l_{1}}^{l_{1}+1}\sum_{i=2}^{k}\abs{\mathcal{F}_{P_{l_{i}}}}.
\]
\end{proof}

The next lemma shows that the $\ceil{\frac{n}{2}}$ bound holds for spiders where the length of each leg is no more than five.  
\begin{lemma}\label{lem:leg_length_at_most5}
Let $S_{l_{1},\ldots,l_{k}}$ be a spider graph where $k\geq 3$ and $1\leq l_{i}\leq 5$ for all $i\in\{1,\ldots,k\}$. 
Then, 
\[
\abs{\mathcal{F}_{S_{l_{1},\ldots,l_{k}}}} \geq \ceil{\frac{n}{2}},
\]
where $n=1+\sum_{i=1}^{k}l_{i}$ denotes the order of the graph. 
\end{lemma}
\begin{proof}
Since $1\leq l_{i}\leq 5$ for all $i\in\{1,\ldots,k\}$, it follows that $k\geq \frac{n-1}{5}$.
Furthermore, by counting the number of minimal forts from~\eqref{eq:mf_spider} that do not include the junction vertex, we obtain the following lower bound:
\begin{align*}
\abs{\mathcal{F}_{S_{l_{1},\ldots,l_{k}}}} \geq \binom{k}{2} &= \frac{k(k-1)}{2} \\
&\geq \frac{(n-1)(n-6)}{50}.
\end{align*}
Note that $\frac{(n-1)(n-6)}{50}\geq\frac{n}{2}$ for all $n\geq 32$.
There are $5305$ non-isomorphic spider graphs with $k\geq 3$ legs and order $4\leq n\leq 31$, where $1\leq l_{i}\leq 5$, for all $i\in\{1,\ldots,k\}$. 
Using the software in~\cite{Cameron2025}, we have verified the $\ceil{\frac{n}{2}}$ bound for all these spider graphs.
The corresponding data is displayed in Figure~\ref{fig:spider_test}.
\end{proof}
\begin{figure}[ht]
\centering 
\includegraphics[width=0.70\textwidth]{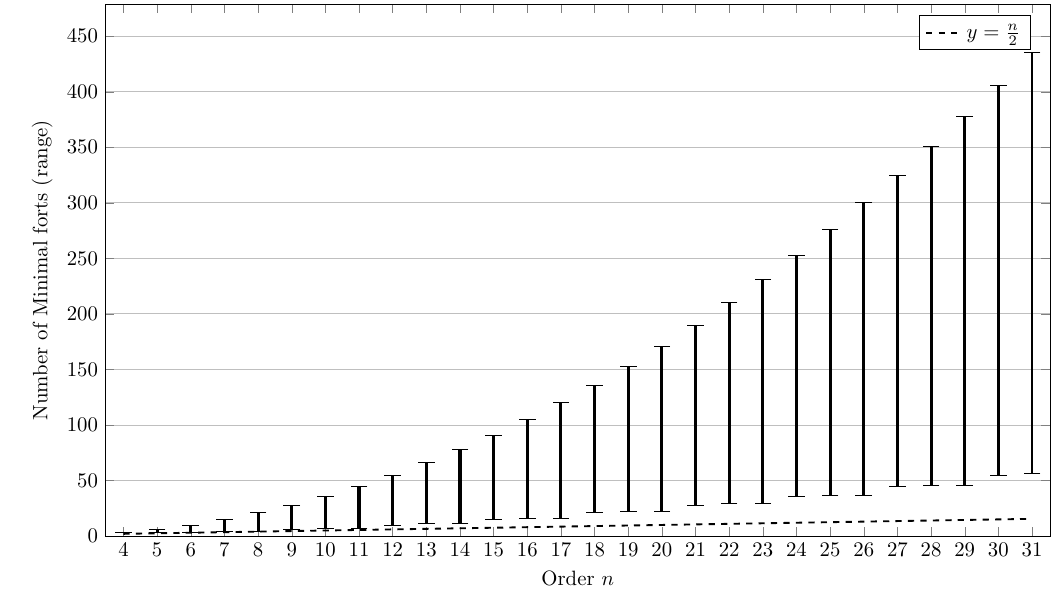}
\caption{For each order $4\leq n\leq 31$, the range of the number of minimal forts is displayed for all spider graphs with $k\geq 3$, where $1\leq l_{i}\leq 5$, for all $i\in\{1,\ldots,k\}$. Also, the bound $y=n/2$ is displayed illustrating that the bound in Lemma~\ref{lem:leg_length_at_most5} holds. The source code is available at~\myurl.}
\label{fig:spider_test}
\end{figure}
\begin{figure}[ht]
\centering
\begin{tikzpicture}[scale=0.65]
\begin{scope}[xshift=0em]
    \node (0) at (0,0)   [nodeFilledwText] {$0$};
    \node (1) at (-1.5,0)[nodeFilledwText] {$1$};
    \node (2) at (-3,0)  [nodeFilledwText] {$2$};
    \node (3) at (-4.5,0)[nodeFilledwText] {$3$};
    \node (4) at (1.5,0) [nodeFilledwText] {$4$};
    \node (5) at (3,0)   [nodeFilledwText] {$5$};
    \node (6) at (4.5,0) [nodeFilledwText] {$6$};
    \node (7) at (0,1.5) [nodeFilledwText] {$7$};
    \draw[lineDecorate] (3)--(2)--(1)--(0)--(4)--(5)--(6);
    \draw[lineDecorate] (7)--(0);
    \node[lab] at (0,-1.0) {Spider $S_{1,3,3}$};
\end{scope}
\begin{scope}[xshift=15em, yshift=5em, scale=0.65]
    \node (0a) at (0,0)    [nodeFilledwoText] {};
    \node (1a) at (-1.5,0) [nodeEmptywoText] {};
    \node (2a) at (-3,0)   [nodeFilledwoText] {};
    \node (3a) at (-4.5,0) [nodeEmptywoText] {};
    \node (4a) at (1.5,0)  [nodeFilledwoText] {};
    \node (5a) at (3,0)    [nodeFilledwoText] {};
    \node (6a) at (4.5,0)  [nodeFilledwoText] {};
    \node (7a) at (0,1.5)  [nodeEmptywoText] {};
    \draw[lineDecorate] (3a)--(2a)--(1a)--(0a)--(4a)--(5a)--(6a);
    \draw[lineDecorate] (7a)--(0a);
    \node[lab, draw=none] at (0,-1.5) {$F_1=\{1,3,7\}$};
\end{scope}
\begin{scope}[xshift=15em, yshift=-5em, scale=0.65]
    \node (0b) at (0,0)    [nodeFilledwoText] {};
    \node (1b) at (-1.5,0) [nodeFilledwoText] {};
    \node (2b) at (-3,0)   [nodeFilledwoText] {};
    \node (3b) at (-4.5,0) [nodeFilledwoText] {};
    \node (4b) at (1.5,0)  [nodeEmptywoText] {};
    \node (5b) at (3,0)    [nodeFilledwoText] {};
    \node (6b) at (4.5,0)  [nodeEmptywoText] {};
    \node (7b) at (0,1.5)  [nodeEmptywoText] {};
    \draw[lineDecorate] (3b)--(2b)--(1b)--(0b)--(4b)--(5b)--(6b);
    \draw[lineDecorate] (7b)--(0b);
    \node[lab, draw=none] at (0,-1.5) {$F_2=\{4,6,7\}$};
\end{scope}
\begin{scope}[xshift=35em, yshift=5em, scale=0.65]
    \node (0c) at (0,0)    [nodeFilledwoText] {};
    \node (1c) at (-1.5,0) [nodeEmptywoText] {};
    \node (2c) at (-3,0)   [nodeFilledwoText] {};
    \node (3c) at (-4.5,0) [nodeEmptywoText] {};
    \node (4c) at (1.5,0)  [nodeEmptywoText] {};
    \node (5c) at (3,0)    [nodeFilledwoText] {};
    \node (6c) at (4.5,0)  [nodeEmptywoText] {};
    \node (7c) at (0,1.5)  [nodeFilledwoText] {};
    \draw[lineDecorate] (3c)--(2c)--(1c)--(0c)--(4c)--(5c)--(6c);
    \draw[lineDecorate] (7c)--(0c);
    \node[lab, draw=none] at (0,-1.5) {$F_3=\{1,3,4,6\}$};
\end{scope}
\begin{scope}[xshift=35em, yshift=-5em, scale=0.65]
    \node (0d) at (0,0)    [nodeEmptywoText] {};
    \node (1d) at (-1.5,0) [nodeFilledwoText] {};
    \node (2d) at (-3,0)   [nodeEmptywoText] {};
    \node (3d) at (-4.5,0) [nodeEmptywoText] {};
    \node (4d) at (1.5,0)  [nodeFilledwoText] {};
    \node (5d) at (3,0)    [nodeEmptywoText] {};
    \node (6d) at (4.5,0)  [nodeEmptywoText] {};
    \node (7d) at (0,1.5)  [nodeEmptywoText] {};
    \draw[lineDecorate] (3d)--(2d)--(1d)--(0d)--(4d)--(5d)--(6d);
    \draw[lineDecorate] (7d)--(0d);
    \node[lab, draw=none] at (0,-1.5) {$F_4=\{0,2,3,5,6,7\}$};
\end{scope}
\end{tikzpicture}
\caption{A spider $S_{1,3,3}$ with exactly $n/2$ minimal forts (left) and the the $4$ minimal forts are displayed in white (right).}
\label{fig:spider_exact_half}
\end{figure}

It is worth noting that the bound in Lemma~\ref{lem:leg_length_at_most5} holds with equality for only two spiders: one of order $6$ and one of order $8$.
The order $8$ spider and its minimal forts are displayed in Figure~\ref{fig:spider_exact_half}. 
The following theorem shows that the $\ceil{\frac{n}{2}}$ bound holds for all spiders.
\begin{theorem}\label{thm:spider_lower_bound}
Let $S_{l_{1},\ldots,l_{k}}$ be a spider graph where $k\geq 3$ and $l_{i}\geq 1$ for all $i\in\{1,\ldots,k\}$. 
Then, 
\[
\abs{\mathcal{F}_{S_{l_{1},\ldots,l_{k}}}} \geq \ceil{\frac{n}{2}},
\]
where $n=1+\sum_{i=1}^{k}l_{i}$ denotes the order of the graph. 
\end{theorem}
\begin{proof}
We proceed via induction on the largest leg length. 
If $1\leq l_{i}\leq 5$, then the bound follows from Lemma~\ref{lem:leg_length_at_most5}. 
Let $L\geq 5$ and suppose that the bound holds for all spiders $S_{l_{1},\ldots,l_{k}}$, where $1\leq l_{i}\leq L$ for all $i\in\{1,\ldots,k\}$. 
Now, consider the spider $S_{l_{1}+1,\ldots,l_{k}}$ that is constructed by appending a vertex to the first leg of $S_{l_{1},\ldots,l_{k}}$.
If $l_{1}<L$, then the bound follows from the induction hypothesis. 
If $l_{1}\geq L\geq 5$, then Lemma~\ref{lem:increase_leg_length1} implies that 
\begin{align*}
\abs{\mathcal{F}_{S_{l_{1}+1,\ldots,l_{k}}}}  &\geq \abs{\mathcal{F}_{S_{l_{1},\ldots,l_{k}}}} + \sum_{i=2}^{k}\abs{\mathcal{F}_{P_{l_{i}}}} \\
&> \abs{\mathcal{F}_{S_{l_{1},\ldots,l_{k}}}} +  1 \\
&\geq \frac{n}{2} + 1 > \frac{n+1}{2}. 
\end{align*}
It follows that the bound holds for all spiders $S_{l_{1},\ldots,l_{k}}$, where $1\leq l_{1}\leq L+1$ and $1\leq l_{i}\leq L$ for all $i\in\{2,\ldots,k\}$. 
Moreover, for each $i\in\{2,\ldots,k\}$, when appending a vertex to leg $i$ to extend its length from $L$ to $L+1$, Lemma~\ref{lem:increase_leg_length1} implies that we gain an additional minimal fort. 
Therefore, the bound holds for all spiders $S_{l_{1},\ldots,l_{k}}$, where $1\leq l_{i}\leq L+1$ for all $i\in\{1,\ldots,k\}$. 
Hence, by the principle of mathematical induction, the bound holds for all spiders. 
\end{proof}

\subsection{Enumerating Minimal Forts with no Star Centers}\label{subsec:mf_no_star_centers_enumerate}
Next, we prove that the $\ceil{\frac{n}{2}}$ lower bound holds for all trees with no star centers. 
Recall that the star centers of a tree are exactly the fort irrelevant vertices of the tree, see Theorem~\ref{thm:tree_star_centers}.
We begin with the following proposition that guarantees the existence of a minimal fort with certain structure when there are no star centers. 
\begin{proposition}\label{prop:deg2_mf}
Let $T$ be a tree with no star centers. 
Let $v\in V(T)$ such that $\deg(v)\geq 2$.
Then, there exists a minimal fort $F$ of $T$ such that $v\notin F$ and $\abs{N(v)\cap F}=2$.
\end{proposition}
\begin{proof}
Let $x,y\in N(v)$.
Let $C_{x}$ and $C_{y}$ denote the connected components of $T-v$ that include $x$ and $y$, respectively. 
Since $T$ has no star centers, the vertices $x$ and $y$ each have at most one pendant neighbor in $T$.
Removing $v$ cannot create additional pendant neighbors for $x$ or $y$; so, neither vertex becomes a star center in $T-v$.
Therefore, there exist minimal forts $F_{x}$ and $F_{y}$ of $C_{x}$ and $C_{y}$, respectively, that include $x$ and $y$, respectively. 
Define $F=F_{x}\cup F_{y}$.
Then, $F$ is a minimal fort of $T$ since it satisfies the combinatorial-cut conditions in Theorem~\ref{thm:mf_comb_cut_char}. 
\end{proof}

The following theorem shows that the $\ceil{\frac{n}{2}}$ bound holds for all trees with no star centers. 
\begin{theorem}\label{thm:zero_star_centers}
Let $T$ be a tree of order $n\geq 6$ with no star centers. 
Then, 
\[
\abs{\mathcal{F}_{T}}\geq\ceil{\frac{n}{2}}.
\]
\end{theorem}
\begin{proof}
We proceed via induction on the number of junction vertices. 
When there are no junction vertices, $T$ is a path graph of order $n\geq 6$. 
Hence, Proposition~\ref{prop:path_sharper_lower_bound} implies that $\abs{\mathcal{F}_{T}}\geq\frac{n}{2}$.
When there is one junction vertex, $T$ is a spider graph of order $n\geq 6$.
Hence, Theorem~\ref{thm:spider_lower_bound} implies that $\abs{\mathcal{F}_{T}}\geq\frac{n}{2}$.

Let $j\geq 1$ and suppose that $\abs{\mathcal{F}_{T}}\geq\frac{n}{2}$ for all trees $T$ of order $n\geq 6$ with no star centers and $j$ or fewer junction vertices. 
Let $T$ be a tree with $(j+1)$ junction vertices. 
Let $u,v$ be junction vertices of $T$ such that the $(u,v)-$path contains no other junction vertices of $T$.
We proceed by cases on the length of the $(u,v)-$path.
In each case, we delete an edge or vertices from the $(u,v)-$path and consider the forts of the components $C_{u}$ and $C_{v}$, which contain the vertices $u$ and $v$, respectively. 
Note that the components $C_{u}$ and $C_{v}$ have at most $j$ junction vertices.

Suppose that the $(u,v)-$path has length $1$. 
Let $C_{u}$ and $C_{v}$ denote the connected components of $T-uv$ that contain $u$ and $v$, respectively. 
In what follows, we show how each fort of $C_{u}$ can be used to construct a minimal fort of $T$.
Each construction can also be applied to minimal forts of $C_{v}$.
Moreover, the constructed minimal forts of $T$ are unique.
Therefore, the induction hypothesis implies that 
\begin{align*}
\abs{\mathcal{F}_{T}} &\geq \abs{\mathcal{F}_{C_{u}}} + \abs{\mathcal{F}_{C_{v}}} \\
&\geq \frac{\abs{V(C_{u})}}{2} + \frac{\abs{V(C_{v})}}{2} = \frac{n}{2}.
\end{align*}
Let $F$ be a minimal fort of $C_{u}$ that does not contain $u$.
Then, $F$ is a minimal fort of $T$ since it satisfies the combinatorial-cut conditions in Theorem~\ref{thm:mf_comb_cut_char}. 
Let $F_{u}$ denote a minimal fort of $C_{u}$ that contains $u$.
By Proposition~\ref{prop:deg2_mf}, there is a minimal fort $F'$ of $C_{v}$ such that $v\notin F'$ and $\abs{N(v)\cap F'}=2$. 
Let $N(v)\cap F'=\{x,y\}$ and define $F'_{x}=F'\cap V(C_{x})$ and $F'_{y}=F'\cap V(C_{y})$, where $C_{x}$ and $C_{y}$ are the components of $C_{v}-v$ that contain $x$ and $y$, respectively. 
Then, $F=F_{u}\cup F'_{x}$ (or $F=F_{u}\cup F'_{y}$) is a minimal fort of $T$ since it satisfies the combinatorial-cut conditions in Theorem~\ref{thm:mf_comb_cut_char}. 

Suppose that the $(u,v)-$path has length $2$.
Let $w$ denote the neighbor of $u$ and $v$ and let $C_{u}$ and $C_{v}$ denote the components $T-w$ that contain $u$ and $v$, respectively. 
As in the previous case, every minimal fort of $C_{u}$ (or $C_{v}$) that does not contain $u$ (or $v$) is a minimal fort of $T$. 
Let $F_{u_{1}},\ldots,F_{u_{k}}$ denote the minimal forts of $C_{u}$ that contain $u$.
Since $C_{u}$ has no star centers, it follows that $k\geq 1$.
Similarly, let $F_{v_{1}},\ldots,F_{v_{l}}$ denote the minimal forts of $C_{v}$ that contain $v$.
Again, since $C_{v}$ has no star centers, it follows that $l\geq 1$. 
Note that the union $F_{u_{i}}\cup F_{v_{j}}$ is a minimal fort of $T$, for any $1\leq i\leq k$ and $1\leq j\leq l$, since it satisfies the combinatorial-cut conditions in Theorem~\ref{thm:mf_comb_cut_char}. 
Since $k,l\geq 1$, we have identified $kl\geq k+l-1$ unique minimal forts of $T$ that include $u$ and $v$, but do not include $w$. 
At this point we have accounted for all but at most one of the minimal forts from $C_{u}$ and $C_{v}$.
By Proposition~\ref{prop:deg2_mf}, there exists minimal forts $F'_{u}$ of $C_{u}$ and $F'_{v}$ of $C_{v}$ such that $u\notin F'_{u}$, $v\notin F'_{v}$, $\abs{N(u)\cap F'_{u}}=2$, and $\abs{N(v)\cap F'_{v}}=2$. 
Let $N(u)\cap F'_{u}=\{u_{x},u_{y}\}$ and define $F'_{u_{x}}=F'_{u}\cap V(C_{u_{x}})$ and $F'_{u_{y}}=F'_{u}\cap V(C_{u_{y}})$, where $C_{u_{x}}$ and $C_{u_{y}}$ are the connected components of $C_{u}-u$ that contain $u_{x}$ and $u_{y}$, respectively. 
Similarly, let $N(v)\cap F'_{v}=\{v_{x},v_{y}\}$ and define $F'_{v_{x}}=F'_{v}\cap V(C_{v_{x}})$ and $F'_{v_{y}}=F'_{v}\cap V(C_{v_{y}})$, where $C_{v_{x}}$ and $C_{v_{y}}$ are the connected components of $C_{v}-v$ that contain $v_{x}$ and $v_{y}$, respectively. 
Then, $F'_{u_{i}}\cup F'_{v_{j}}\cup\{w\}$ is a minimal fort of $T$, for any $i,j\in\{x,y\}$. 
Hence, we have identified $4$ unique minimal forts of $T$ that do not include $u$ or $v$, but do include $w$.
Therefore, we have 
\begin{align*}
\abs{\mathcal{F}_{T}} &\geq \abs{\mathcal{F}_{C_{u}}} + \abs{\mathcal{F}_{C_{v}}} + 1 \\
&\geq \frac{\abs{V(C_{u})}}{2} + \frac{\abs{V(C_{v})}}{2} + 1 \\
&= \frac{n-1}{2} + 1 = \frac{n+1}{2} \geq \frac{n}{2}.
\end{align*}

Suppose that the $(u,v)-$path has length $l\geq 3$.
Let $P$ denote the $(u,v)-$path and let $T-\left(V(P)\setminus{\{u,v\}}\right)$ denote the graph induced by deleting all vertices in the path $P$ from $T$, except for $u$ and $v$. 
Let $C_{u}$ and $C_{v}$ denote the connected components of $T-\left(V(P)\setminus{\{u,v\}}\right)$ that contain $u$ and $v$, respectively. 
As in the previous cases, every minimal fort of $C_{u}$ (or $C_{v}$) that does not contain $u$ (or $v$) is a minimal fort of $T$. 
Let $F_{u}$ denote a minimal fort of $C_{u}$ that contains $u$. 
By Proposition~\ref{prop:deg2_mf}, there is a minimal fort $F'$ of $C_{v}$ such that $v\notin F'$ and $\abs{N(v)\cap F'}=2$. 
Let $N(v)\cap F'=\{x,y\}$ and define $F'_{x}=F'\cap V(C_{x})$ and $F'_{y}=F'\cap V(C_{y})$, where $C_{x}$ and $C_{y}$ are the components of $C_{v}-v$ that contain $x$ and $y$, respectively.
Also, let the $(u,v)-$path be written as 
\[
P = u + u' + P_{l-2} + v,
\]
where the $+$ denotes path concatenation and $P_{l-2}$ denotes a path of order $l-2\geq 1$. 
Let $\hat{F}$ denote a minimal fort of $P_{l-2}$.
Then, $F= F_{u}\cup\hat{F}\cup F'_{x}$ (or $F=F_{u}\cup\hat{F}\cup F'_{y}$) is a minimal fort of $T$ since it satisfies the combinatorial-cut conditions in Theorem~\ref{thm:mf_comb_cut_char}. 
Therefore, each minimal fort of $C_{u}$ that contains $u$ corresponds to at least $2\abs{\mathcal{F}_{P_{l-2}}}$ minimal forts of $T$. 
Similarly, each minimal fort of $C_{v}$ that contains $v$ corresponds to at least $2\abs{\mathcal{F}_{P_{l-2}}}$ minimal forts of $T$. 
Note that there is no double counting since the branches $F'_{x}$ (or $F'_{y}$) are not minimal forts by themselves in their respective components. 
Furthermore, since there are no star centers, both $C_{u}$ and $C_{v}$ contain at least one minimal fort that contains $u$ and $v$, respectively. 
Therefore, we have 
\begin{align*}
\abs{\mathcal{F}_{T}} &\geq \abs{\mathcal{F}_{C_{u}}} + \abs{\mathcal{F}_{C_{v}}} + 2\abs{\mathcal{F}_{P_{l-2}}} \\
&\geq \frac{\abs{V(C_{u})}}{2} + \frac{\abs{V(C_{v})}}{2} + 2\abs{\mathcal{F}_{P_{l-2}}} \\
&= \frac{n-(l-1)}{2} + 2\abs{\mathcal{F}_{P_{l-2}}} \\
&= \frac{n-l+1+4\abs{\mathcal{F}_{P_{l-2}}}}{2} \ge \frac{n}{2}. 
\end{align*}
\end{proof}

In what follows, we illustrate the minimal fort constructions utilized in the proof of Theorem~\ref{thm:zero_star_centers}.
In Figure~\ref{fig:thm47_l1_constructions}, we show the two types of fort constructions used in the case where the distance between the junction vertices $u$ and $v$ is $l=1$. 
In Figure~\ref{fig:thm47_l2_constructions}, we show the three types of fort constructions used in the case where the distance between the junction vertices $u$ and $v$ is $l=2$.
In Figure~\ref{fig:thm47_l3_constructions}, we show the two types of fort constructions used in the case where the distance between the junction vertices $u$ and $v$ is $l\geq 3$.
It is easy to verify that all these minimal fort constructions satisfy the combinatorial-cut conditions in Theorem~\ref{thm:mf_comb_cut_char}.
\begin{figure}[ht]
\centering
\begin{tikzpicture}[scale=0.65]
\begin{scope}[xshift=0em]
    \node (u)  at (-2.6,0) [nodeFilledwText] {$u$};
    \node (v)  at ( 2.6,0) [nodeFilledwText] {$v$};
    \node (ux)  at (-4.0, 1.1) [nodeFilledwText] {$u_x$};
    \node (uxx) at (-5.1, 2.0) [nodeFilledwText] {$u_x'$};
    \node (uy)  at (-4.0,-1.1) [nodeFilledwText] {$u_y$};
    \node (uyy) at (-5.1,-2.0) [nodeFilledwText] {$u_y'$};
    \node (vx)  at (4.0, 1.1) [nodeFilledwText] {$v_x$};
    \node (vxx) at (5.1, 2.0) [nodeFilledwText] {$v_x'$};
    \node (vy)  at (4.0,-1.1) [nodeFilledwText] {$v_y$};
    \node (vyy) at (5.1,-2.0) [nodeFilledwText] {$v_y'$};
    \draw[lineDecorate] (u)--(v);
    \draw[lineDecorate] (u)--(ux)--(uxx);
    \draw[lineDecorate] (u)--(uy)--(uyy);
    \draw[lineDecorate] (v)--(vx)--(vxx);
    \draw[lineDecorate] (v)--(vy)--(vyy);
    \node[lab] at (0,-4.0) {Tree $T$ with junction vertices a distance of $1$ apart};
\end{scope}
\begin{scope}[xshift=30em, yshift=7em, scale=0.65]
    \node (u)  at (-2.6,0) [nodeFilledwoText] {};
    \node (v)  at ( 2.6,0) [nodeFilledwoText] {};
    \node (ux)  at (-4.0, 1.1) [nodeEmptywoText] {};
    \node (uxx) at (-5.1, 2.0) [nodeEmptywoText] {};
    \node (uy)  at (-4.0,-1.1) [nodeEmptywoText] {};
    \node (uyy) at (-5.1,-2.0) [nodeEmptywoText] {};
    \node (vx)  at (4.0, 1.1) [nodeFilledwoText] {};
    \node (vxx) at (5.1, 2.0) [nodeFilledwoText] {};
    \node (vy)  at (4.0,-1.1) [nodeFilledwoText] {};
    \node (vyy) at (5.1,-2.0) [nodeFilledwoText] {};
    \draw[lineDecorate] (u)--(v);
    \draw[lineDecorate] (u)--(ux)--(uxx);
    \draw[lineDecorate] (u)--(uy)--(uyy);
    \draw[lineDecorate] (v)--(vx)--(vxx);
    \draw[lineDecorate] (v)--(vy)--(vyy);
    \node[lab,align=center] at (0,-4.0)
        {Type 1: $F\in\mathcal{F}_{C_u}$ with $u\notin F$};
\end{scope}
\begin{scope}[xshift=30em, yshift=-7em, scale=0.65]
    \node (u)  at (-2.6,0) [nodeEmptywoText] {};
    \node (v)  at ( 2.6,0) [nodeFilledwoText] {};
    \node (ux)  at (-4.0, 1.1) [nodeFilledwoText] {};
    \node (uxx) at (-5.1, 2.0) [nodeEmptywoText] {};
    \node (uy)  at (-4.0,-1.1) [nodeFilledwoText] {};
    \node (uyy) at (-5.1,-2.0) [nodeEmptywoText] {};
    \node (vx)  at (4.0, 1.1) [nodeFilledwoText] {};
    \node (vxx) at (5.1, 2.0) [nodeFilledwoText] {};
    \node (vy)  at (4.0,-1.1) [nodeEmptywoText] {};
    \node (vyy) at (5.1,-2.0) [nodeEmptywoText] {};
    \draw[lineDecorate] (u)--(v);
    \draw[lineDecorate] (u)--(ux)--(uxx);
    \draw[lineDecorate] (u)--(uy)--(uyy);
    \draw[lineDecorate] (v)--(vx)--(vxx);
    \draw[lineDecorate] (v)--(vy)--(vyy);
    \node[lab,align=center] at (0,-4.0)
        {Type 2: $F_u\cup F'_y$\\
        with $u\in F_u$ and branch $F'_y$};
\end{scope}
\end{tikzpicture}
\caption{Illustration of the two constructions used in the case $l=1$ of Theorem~\ref{thm:zero_star_centers}. Fort vertices are white.}
\label{fig:thm47_l1_constructions}
\end{figure}
\begin{figure}[ht]
\centering
\begin{tikzpicture}[scale=0.70]
\begin{scope}[xshift=0em]
    \node (u)  at (-2.6,0) [nodeFilledwText] {$u$};
    \node (w)  at ( 0.0,0) [nodeFilledwText] {$w$};
    \node (v)  at ( 2.6,0) [nodeFilledwText] {$v$};
    \node (ux)  at (-4.0, 1.1) [nodeFilledwText] {$u_x$};
    \node (uxx) at (-5.1, 2.0) [nodeFilledwText] {$u_x'$};
    \node (uy)  at (-4.0,-1.1) [nodeFilledwText] {$u_y$};
    \node (uyy) at (-5.1,-2.0) [nodeFilledwText] {$u_y'$};
    \node (vx)  at (4.0, 1.1) [nodeFilledwText] {$v_x$};
    \node (vxx) at (5.1, 2.0) [nodeFilledwText] {$v_x'$};
    \node (vy)  at (4.0,-1.1) [nodeFilledwText] {$v_y$};
    \node (vyy) at (5.1,-2.0) [nodeFilledwText] {$v_y'$};
    \draw[lineDecorate] (u)--(w)--(v);
    \draw[lineDecorate] (u)--(ux)--(uxx);
    \draw[lineDecorate] (u)--(uy)--(uyy);
    \draw[lineDecorate] (v)--(vx)--(vxx);
    \draw[lineDecorate] (v)--(vy)--(vyy);
    \node[lab] at (0,-4.0) {Tree $T$ with junction vertices a distance of $2$ apart};
\end{scope}
\begin{scope}[xshift=30em, yshift=14em, scale=0.65]
    \node (u)  at (-2.6,0) [nodeFilledwoText] {};
    \node (w)  at ( 0.0,0) [nodeFilledwoText] {};
    \node (v)  at ( 2.6,0) [nodeFilledwoText] {};
    \node (ux)  at (-4.0, 1.1) [nodeEmptywoText] {};
    \node (uxx) at (-5.1, 2.0) [nodeEmptywoText] {};
    \node (uy)  at (-4.0,-1.1) [nodeEmptywoText] {};
    \node (uyy) at (-5.1,-2.0) [nodeEmptywoText] {};
    \node (vx)  at (4.0, 1.1) [nodeFilledwoText] {};
    \node (vxx) at (5.1, 2.0) [nodeFilledwoText] {};
    \node (vy)  at (4.0,-1.1) [nodeFilledwoText] {};
    \node (vyy) at (5.1,-2.0) [nodeFilledwoText] {};
    \draw[lineDecorate] (u)--(w)--(v);
    \draw[lineDecorate] (u)--(ux)--(uxx);
    \draw[lineDecorate] (u)--(uy)--(uyy);
    \draw[lineDecorate] (v)--(vx)--(vxx);
    \draw[lineDecorate] (v)--(vy)--(vyy);
    \node[lab,align=center] at (0,-4.0)
        {Type 1: $F\in\mathcal{F}_{C_u}$ with $u\notin F$};
\end{scope}
\begin{scope}[xshift=30em, yshift=0em, scale=0.65]
    \node (u)  at (-2.6,0) [nodeEmptywoText] {};
    \node (w)  at ( 0.0,0) [nodeFilledwoText] {};
    \node (v)  at ( 2.6,0) [nodeEmptywoText] {};
    \node (ux)  at (-4.0, 1.1) [nodeFilledwoText] {};
    \node (uxx) at (-5.1, 2.0) [nodeEmptywoText] {};
    \node (uy)  at (-4.0,-1.1) [nodeFilledwoText] {};
    \node (uyy) at (-5.1,-2.0) [nodeEmptywoText] {};
    \node (vx)  at (4.0, 1.1) [nodeFilledwoText] {};
    \node (vxx) at (5.1, 2.0) [nodeEmptywoText] {};
    \node (vy)  at (4.0,-1.1) [nodeFilledwoText] {};
    \node (vyy) at (5.1,-2.0) [nodeEmptywoText] {};
    \draw[lineDecorate] (u)--(w)--(v);
    \draw[lineDecorate] (u)--(ux)--(uxx);
    \draw[lineDecorate] (u)--(uy)--(uyy);
    \draw[lineDecorate] (v)--(vx)--(vxx);
    \draw[lineDecorate] (v)--(vy)--(vyy);
    \node[lab,align=center] at (0,-4.0)
        {Type 2: $F_{u_i}\cup F_{v_j}$\\
        with $u\in F_{u_i}$ and $v\in F_{v_j}$};
\end{scope}
\begin{scope}[xshift=30em, yshift=-14em, scale=0.65]
    \node (u)  at (-2.6,0) [nodeFilledwoText] {};
    \node (w)  at ( 0.0,0) [nodeEmptywoText] {};
    \node (v)  at ( 2.6,0) [nodeFilledwoText] {};
    \node (ux)  at (-4.0, 1.1) [nodeEmptywoText] {};
    \node (uxx) at (-5.1, 2.0) [nodeEmptywoText] {};
    \node (uy)  at (-4.0,-1.1) [nodeFilledwoText] {};
    \node (uyy) at (-5.1,-2.0) [nodeFilledwoText] {};
    \node (vx)  at (4.0, 1.1) [nodeEmptywoText] {};
    \node (vxx) at (5.1, 2.0) [nodeEmptywoText] {};
    \node (vy)  at (4.0,-1.1) [nodeFilledwoText] {};
    \node (vyy) at (5.1,-2.0) [nodeFilledwoText] {};
    \draw[lineDecorate] (u)--(w)--(v);
    \draw[lineDecorate] (u)--(ux)--(uxx);
    \draw[lineDecorate] (u)--(uy)--(uyy);
    \draw[lineDecorate] (v)--(vx)--(vxx);
    \draw[lineDecorate] (v)--(vy)--(vyy);
    \node[lab,align=center] at (0,-4.0)
        {Type 3: $F'_{u_i}\cup F'_{v_j}\cup\{w\}$\\
        with branches $F'_{u_i},F'_{v_j}$};
\end{scope}
\end{tikzpicture}
\caption{Illustration of the three constructions used in the case $l=2$ of Theorem~\ref{thm:zero_star_centers}. Fort vertices are white.}
\label{fig:thm47_l2_constructions}
\end{figure}
\begin{figure}[ht]
\centering
\begin{tikzpicture}[scale=0.65]
\begin{scope}[xshift=0em]
    \node (u)  at (-3.0,0) [nodeFilledwText] {$u$};
    \node (w)  at (-1.8,0) [nodeFilledwText] {$w$};
    \node (p1)  at (-0.6,0) [nodeFilledwText] {$p_1$};
    \node (p2)  at ( 0.6,0) [nodeFilledwText] {$p_2$};
    \node (p3)  at ( 1.8,0) [nodeFilledwText] {$p_3$};
    \node (v)  at ( 3.0,0) [nodeFilledwText] {$v$};
    \node (ux)  at (-4.4, 1.1) [nodeFilledwText] {$u_x$};
    \node (uxx) at (-5.5, 2.0) [nodeFilledwText] {$u_x'$};
    \node (uy)  at (-4.4,-1.1) [nodeFilledwText] {$u_y$};
    \node (uyy) at (-5.5,-2.0) [nodeFilledwText] {$u_y'$};
    \node (vx)  at (4.4, 1.1) [nodeFilledwText] {$v_x$};
    \node (vxx) at (5.5, 2.0) [nodeFilledwText] {$v_x'$};
    \node (vy)  at (4.4,-1.1) [nodeFilledwText] {$v_y$};
    \node (vyy) at (5.5,-2.0) [nodeFilledwText] {$v_y'$};
    \draw[lineDecorate] (u)--(w)--(p1)--(p2)--(p3)--(v);
    \draw[lineDecorate] (u)--(ux)--(uxx);
    \draw[lineDecorate] (u)--(uy)--(uyy);
    \draw[lineDecorate] (v)--(vx)--(vxx);
    \draw[lineDecorate] (v)--(vy)--(vyy);
    \node[lab] at (0,-4.0) {Tree $T$ with junction vertices a distance of at least $3$ apart};
\end{scope}
\begin{scope}[xshift=30em, yshift=7em, scale=0.65]
    \node (u)  at (-3.0,0) [nodeFilledwoText] {};
    \node (w)  at (-1.8,0) [nodeFilledwoText] {};
    \node (p1)  at (-0.6,0) [nodeFilledwoText] {};
    \node (p2)  at ( 0.6,0) [nodeFilledwoText] {};
    \node (p3)  at ( 1.8,0) [nodeFilledwoText] {};
    \node (v)  at ( 3.0,0) [nodeFilledwoText] {};
    \node (ux)  at (-4.4, 1.1) [nodeEmptywoText] {};
    \node (uxx) at (-5.5, 2.0) [nodeEmptywoText] {};
    \node (uy)  at (-4.4,-1.1) [nodeEmptywoText] {};
    \node (uyy) at (-5.5,-2.0) [nodeEmptywoText] {};
    \node (vx)  at (4.4, 1.1) [nodeFilledwoText] {};
    \node (vxx) at (5.5, 2.0) [nodeFilledwoText] {};
    \node (vy)  at (4.4,-1.1) [nodeFilledwoText] {};
    \node (vyy) at (5.5,-2.0) [nodeFilledwoText] {};
    \draw[lineDecorate] (u)--(w)--(p1)--(p2)--(p3)--(v);
    \draw[lineDecorate] (u)--(ux)--(uxx);
    \draw[lineDecorate] (u)--(uy)--(uyy);
    \draw[lineDecorate] (v)--(vx)--(vxx);
    \draw[lineDecorate] (v)--(vy)--(vyy);
    \node[lab,align=center] at (0,-4.0)
        {Type 1: $F\in\mathcal{F}_{C_u}$ with $u\notin F$};
\end{scope}
\begin{scope}[xshift=30em, yshift=-7em, scale=0.65]
    \node (u)  at (-3.0,0) [nodeEmptywoText] {};
    \node (w)  at (-1.8,0) [nodeFilledwoText] {};
    \node (p1)  at (-0.6,0) [nodeEmptywoText] {};
    \node (p2)  at ( 0.6,0) [nodeFilledwoText] {};
    \node (p3)  at ( 1.8,0) [nodeEmptywoText] {};
    \node (v)  at ( 3.0,0) [nodeFilledwoText] {};
    \node (ux)  at (-4.4, 1.1) [nodeFilledwoText] {};
    \node (uxx) at (-5.5, 2.0) [nodeEmptywoText] {};
    \node (uy)  at (-4.4,-1.1) [nodeFilledwoText] {};
    \node (uyy) at (-5.5,-2.0) [nodeEmptywoText] {};
    \node (vx)  at (4.4, 1.1) [nodeFilledwoText] {};
    \node (vxx) at (5.5, 2.0) [nodeFilledwoText] {};
    \node (vy)  at (4.4,-1.1) [nodeEmptywoText] {};
    \node (vyy) at (5.5,-2.0) [nodeEmptywoText] {};
    \draw[lineDecorate] (u)--(w)--(p1)--(p2)--(p3)--(v);
    \draw[lineDecorate] (u)--(ux)--(uxx);
    \draw[lineDecorate] (u)--(uy)--(uyy);
    \draw[lineDecorate] (v)--(vx)--(vxx);
    \draw[lineDecorate] (v)--(vy)--(vyy);
    \node[lab,align=center] at (0,-4.0)
        {Type 2: $F_u\cup\hat{F}\cup F'_y$ \\
        with $u\in F_u$, $\hat{F}\in\mathcal{F}(P_{\ell-2})$, and branch $F'_y$};
\end{scope}
\end{tikzpicture}
\caption{Illustration of the two constructions used in the case $l \geq 3$ of Theorem~\ref{thm:zero_star_centers}. Fort vertices are white.}
\label{fig:thm47_l3_constructions}
\end{figure}

\subsection{Enumerating Minimal Forts with Star Centers}\label{subsec:mf_with_star_centers_enumerate}
Next, we prove a useful relationship between the number of minimal forts and the number of star centers. 
Recall that the star centers of a graph are defined recursively. 
In particular, let $T_{0}=T$ and define $S_{0}$ to be the vertices of $T_{0}$ that have $2$ or more pendant neighbors. 
Then, define $T_{1}$ to be the graph formed by removing all vertices in $S_{0}$ from $T_{0}$ as well as their pendant neighbors.
Note that $T_{1}$ may be a forest since it could be disconnected. 
Continuing for $i\geq 1$, let $S_{i}$ be the vertices of $T_{i}$ that have two or more pendant neighbors.
Then, define $T_{i+1}$ to be the graph formed by removing all vertices in $S_{i}$ as well as their pendant neighbors. 
Since the graph $T$ is finite, there exists a $t\geq 0$ such that $S_{t}$ is empty.
We let $\mathcal{S}_{T}$ denote the star centers of $T$, which is defined by the union 
\[
\mathcal{S}_{T}=S_{0}\cup S_{1}\cup\cdots\cup S_{t}. 
\]
\begin{theorem}\label{thm:fort_star_bound}
Let $T$ be a tree of order $n\geq 6$. 
Then, 
\[
2\abs{\mathcal{F}_{T}} + \abs{\mathcal{S}_{T}} \geq n. 
\]
\end{theorem}
\begin{proof}
We proceed via induction on the number of star centers of $T$.
The base case, when $\abs{\mathcal{S}_{T}}=0$, is covered by Theorem~\ref{thm:zero_star_centers}. 
Let $k\geq 0$ and assume the result holds for all trees with $k$ star centers.
Let $T$ be a tree with $k+1$ star centers. 
Then, the star centers of $T$ can be written as 
\[
\mathcal{S}_{T}=S_{0}\cup S_{1}\cup\cdots\cup S_{t},
\]
where $S_{0}\neq\emptyset$ and $t\geq 0$. 
Let $s\in S_{0}$ and let 
\[
N(s)=\left\{u_{1},\ldots,u_{p},v_{1},\ldots,v_{q}\right\},
\]
where $\deg(u_{i})=1$ for $1\leq i\leq p$, $\deg(v_{j})\geq 2$ for $1\leq j\leq q$, $p\geq 2$, and $q\geq 0$. 
If $q=0$, then $T$ is the star graph and the result holds.

Hence, we assume that $q\geq 1$.
Define $T'=T-\left\{s,u_{1},\ldots,u_{p}\right\}$. 
Note that $T'$ is made up of the connected components $C_{1},\ldots,C_{q}$, where $C_{j}$ contains $v_{j}$ for $1\leq j\leq q$. 
Moreover, $T'$ has $k$ star centers and the induction hypothesis implies that the result holds for $T'$.
Let $F'$ denote a minimal fort of $C_{j}$.
If $v_{j}\notin F'$, then $F'$ is a minimal fort of $T$.
If $v_{j}\in F'$, then $F=F'\cup\{u_{i}\}$ is a minimal fort of $T$ for all $1\leq i\leq p$.
Therefore, every minimal fort of $C_{j}$ corresponds to at least one distinct minimal fort of $T$.
The minimal forts of $T'$ are constructed as the union of the minimal forts over each connected component $C_{j}$, for $1\leq j\leq q$. 
Hence, every minimal fort of $T'$ corresponds to at least one distinct minimal fort of $T$.
Furthermore, $T$ has an additional $\binom{p}{2}$ minimal forts from the pendant vertices $u_{1},\ldots,u_{p}$. 
Therefore, we have 
\begin{align*}
2\abs{\mathcal{F}_{T}} + \abs{\mathcal{S}_{T}} &\geq 2\left(\abs{\mathcal{F}_{T'}} + \binom{p}{2}\right) + \abs{\mathcal{S}_{T}} \\
&= \left(2\abs{\mathcal{F}_{T'}} + \abs{\mathcal{S}_{T'}}\right) + p(p-1) + 1 \\
&\geq (n-p-1) + p(p-1) + 1 \\
&\geq (n-p-1) + p + 1 = n. 
\end{align*}
\end{proof}

In Figure~\ref{fig:thm48_mf_constructions}, we illustrate the minimal fort constructions utilized in the proof of Theorem~\ref{thm:fort_star_bound}.
On the left, $T$ is a tree with star center $s$ that has pendant neighbors $u_1,u_2,u_3$ and non-pendant neighbors $v_1,v_2$.
The component $C_1$ at $v_1$ is a spider with three legs of length $2$, and the component $C_2$ at $v_2$ is a path $P_3$ with $v_2$ as a leaf. 
On the right, Type~1 shows a minimal fort from the spider avoiding $v_1$, while Type~2 shows a minimal fort of $P_3$ containing $v_2$ extended by a pendant neighbor of $s$.
\begin{figure}[ht]
\centering
\begin{tikzpicture}[scale=0.65]
\begin{scope}[xshift=0em]
    \node (s)  at (0.0,0.0)   [nodeFilledwText] {$s$};
    \node (u1) at (-1.5,1.5) [nodeFilledwText] {$u_1$};
    \node (u2) at (-2.121,0) [nodeFilledwText] {$u_2$};
    \node (u3) at (-1.5,-1.5) [nodeFilledwText] {$u_3$};
    \node (v1) at (1.5,1.5) [nodeFilledwText] {$v_1$};
    \node (v2) at (1.5,-1.5) [nodeFilledwText] {$v_2$};
    \node (a1) at (3.0,3.0)  [nodeFilledwText] {$a_1$};
    \node (a2) at (4.5,4.5)  [nodeFilledwText] {$a_2$};
    \node (b1) at (3.621,1.5)  [nodeFilledwText] {$b_1$};
    \node (b2) at (5.743,1.5)  [nodeFilledwText] {$b_2$};
    \node (c1) at (3.0,0.0)  [nodeFilledwText] {$c_1$};
    \node (c2) at (4.5,-1.5)  [nodeFilledwText] {$c_2$};
    \node (s')  at (3.0,-3.0)  [nodeFilledwText] {$s'$};
    \node (v1')  at (4.5,-4.5)  [nodeFilledwText] {$v'_1$};
    \draw[lineDecorate] (s)--(u1);
    \draw[lineDecorate] (s)--(u2);
    \draw[lineDecorate] (s)--(u3);
    \draw[lineDecorate] (s)--(v1);
    \draw[lineDecorate] (s)--(v2);
    \draw[lineDecorate] (v1)--(a1)--(a2);
    \draw[lineDecorate] (v1)--(b1)--(b2);
    \draw[lineDecorate] (v1)--(c1)--(c2);
    \draw[lineDecorate] (v2)--(s')--(v1');
    \node[lab,align=center] at (-1.0,-4.0)
        {Tree $T$ with star center $s$};
\end{scope}
\begin{scope}[xshift=25em, yshift=10em, scale=0.65]
    \node (s)  at (0.0,0.0)   [nodeFilledwoText] {};
    \node (u1) at (-1.5,1.5) [nodeFilledwoText] {};
    \node (u2) at (-2.121,0) [nodeFilledwoText] {};
    \node (u3) at (-1.5,-1.5) [nodeFilledwoText] {};
    \node (v1) at (1.5,1.5) [nodeFilledwoText] {};
    \node (v2) at (1.5,-1.5) [nodeFilledwoText] {};
    \node (a1) at (3.0,3.0)  [nodeEmptywoText] {};
    \node (a2) at (4.5,4.5)  [nodeEmptywoText] {};
    \node (b1) at (3.621,1.5)  [nodeEmptywoText] {};
    \node (b2) at (5.743,1.5)  [nodeEmptywoText] {};
    \node (c1) at (3.0,0.0)  [nodeFilledwoText] {};
    \node (c2) at (4.5,-1.5)  [nodeFilledwoText] {};
    \node (s')  at (3.0,-3.0)  [nodeFilledwoText] {};
    \node (v1')  at (4.5,-4.5)  [nodeFilledwoText] {};
    \draw[lineDecorate] (s)--(u1);
    \draw[lineDecorate] (s)--(u2);
    \draw[lineDecorate] (s)--(u3);
    \draw[lineDecorate] (s)--(v1);
    \draw[lineDecorate] (s)--(v2);
    \draw[lineDecorate] (v1)--(a1)--(a2);
    \draw[lineDecorate] (v1)--(b1)--(b2);
    \draw[lineDecorate] (v1)--(c1)--(c2);
    \draw[lineDecorate] (v2)--(s')--(v1');
    \node[lab,align=center] at (-1.0,-4.0)
        {Type 1: $F\in\mathcal{F}_{C_{1}}$ \\
        with $v_{1}\notin F$};
\end{scope}
\begin{scope}[xshift=25em, yshift=-10em, scale=0.65]
    \node (s)  at (0.0,0.0)   [nodeFilledwoText] {};
    \node (u1) at (-1.5,1.5) [nodeEmptywoText] {};
    \node (u2) at (-2.121,0) [nodeFilledwoText] {};
    \node (u3) at (-1.5,-1.5) [nodeFilledwoText] {};
    \node (v1) at (1.5,1.5) [nodeFilledwoText] {};
    \node (v2) at (1.5,-1.5) [nodeEmptywoText] {};
    \node (a1) at (3.0,3.0)  [nodeFilledwoText] {};
    \node (a2) at (4.5,4.5)  [nodeFilledwoText] {};
    \node (b1) at (3.621,1.5)  [nodeFilledwoText] {};
    \node (b2) at (5.743,1.5)  [nodeFilledwoText] {};
    \node (c1) at (3.0,0.0)  [nodeFilledwoText] {};
    \node (c2) at (4.5,-1.5)  [nodeFilledwoText] {};
    \node (s')  at (3.0,-3.0)  [nodeFilledwoText] {};
    \node (v1')  at (4.5,-4.5)  [nodeEmptywoText] {};
    \draw[lineDecorate] (s)--(u1);
    \draw[lineDecorate] (s)--(u2);
    \draw[lineDecorate] (s)--(u3);
    \draw[lineDecorate] (s)--(v1);
    \draw[lineDecorate] (s)--(v2);
    \draw[lineDecorate] (v1)--(a1)--(a2);
    \draw[lineDecorate] (v1)--(b1)--(b2);
    \draw[lineDecorate] (v1)--(c1)--(c2);
    \draw[lineDecorate] (v2)--(s')--(v1');
    \node[lab,align=center] at (-1.0,-4.0)
        {Type 2: $F'\cup\{u_i\}$ \\
        with $F'\in\mathcal{F}_{P_3}$ and $v_2\in F'$};
\end{scope}
\end{tikzpicture}
\caption{Illustration of the two constructions used in Theorem~\ref{thm:fort_star_bound}. Fort vertices are white.}
\label{fig:thm48_mf_constructions}
\end{figure}

The following lemma puts an upper bound on the number of star centers in any graph. 
\begin{lemma}\label{lem:star_center_bound}
Let $G$ be a graph of order $n\geq 1$ and let $\mathcal{S}_{G}$ denote the star centers of $G$. 
Then,
\[
\abs{\mathcal{S}_{G}} \leq \floor{\frac{n}{3}}.
\]
\end{lemma}
\begin{proof}
At each stage $i$, every vertex $v\in S_{i}$ has at least two pendant neighbors in $G_{i}$.
When we construct $G_{i+1}$ from $G_{i}$, we delete all vertices in $S_{i}$ from $G_{i}$ as well as their pendant neighbors. 
Note that the pendant neighbors are disjoint, that is, if $u,v\in S_{i}$ are distinct, then the pendant neighbors of $u$ and $v$ in $G_{i}$ are disjoint. 
Therefore, at stage $i$, the number of deleted vertices is at least 
\[
\abs{S_{i}} + 2\abs{S_{i}} = 3\abs{S_{i}}. 
\]
Since a vertex cannot be deleted more than once, it follows that 
\[
n \geq 3 \left(\abs{S_{0}} + \abs{S_{1}} + \cdots + \abs{S_{t}}\right) = 3\abs{\mathcal{S}_{G}}. 
\]
\end{proof}

Theorem~\ref{thm:fort_star_bound} and Lemma~\ref{lem:star_center_bound} imply a lower bound on the number of minimal forts of any tree of order at least six. 
We state this bound in the following corollary. 
\begin{corollary}\label{cor:minimal_fort_bound}
Let $T$ be a tree of order $n\geq 1$. 
Then, 
\[
\abs{\mathcal{F}_{T}}\geq \ceil{\frac{n}{3}}. 
\]
\end{corollary}
\begin{proof}
This result can be easily verified for trees of order $1\leq n\leq 5$.
For $n\geq 6$, this result follows from Theorem~\ref{thm:fort_star_bound} and Lemma~\ref{lem:star_center_bound}. 
\end{proof}

\section{Trees with the Fewest Possible Minimal Forts}\label{sec:equiv_thm}
In this section, we characterize the trees that attain the lower bound in Corollary~\ref{cor:minimal_fort_bound}. 
The following lemma shows that the lower bound is not sharp for trees with star centers that induce a disconnected graph. 
\begin{lemma}\label{lem:minimal_fort_bound+1}
Let $T$ be a tree of order $n\geq 1$.
Suppose that $\mathcal{S}_{T}\neq\emptyset$ and $T[\mathcal{S}_{T}]$ is disconnected.
Then, 
\[
\abs{\mathcal{F}_{T}} > \frac{n}{3}. 
\]
\end{lemma}
\begin{proof}
Since $\mathcal{S}_{T}$ is non-empty and $T[\mathcal{S}_{T}]$ is disconnected, it follows that there exists an $s\in S_{0}$ such that $s$ has a neighbor that is neither a pendant vertex nor a star center. 
Let
\[
N(s)=\left\{u_{1},\ldots,u_{p},v_{1},\ldots,v_{q}\right\},
\]
where $\deg(u_{i})=1$ for $1\leq i\leq p$, $\deg(v_{j})\geq 2$ for $1\leq j\leq q$, $p\geq 2$, and $q\geq 1$. 
Define $T'=T-\{s,u_{1},\ldots,u_{p}\}$. 
Note that $T'$ is made up of the connected components $C_{1},\ldots,C_{q}$, where $C_{j}$ contains $v_{j}$ for $1\leq j\leq q$. 
Since the labeling is arbitrary and $s$ is adjacent to at least one vertex that is not a pendant nor a star center, we can assume that $v_{1}$ is not a star center of $T$. 
In fact, since every star center of $T'$ is a star center of $T$, we can assume that $v_{1}$ is not a star center of $T'$.

Let $F'$ be a minimal fort of $C_{j}$, for any $j\in\{1,\ldots,q\}$. 
If $v_{j}\notin F'$, then $F'$ is a minimal fort of $T$.
If $v_{j}\in F'$, then $F=F'\cup\{u_{i}\}$ is a minimal fort for any $i\in\{1,\ldots,p\}$.
Since $v_{1}$ is not a star center of $T'$, Theorem~\ref{thm:tree_star_centers} implies that there exists at least one minimal fort of $C_{1}$ that includes $v_{1}$. 
Thus, since $p\geq 2$, $T$ has at least $\abs{\mathcal{F}_{C_{1}}}+1$ minimal forts corresponding to the minimal forts of $C_{1}$. 
Moreover, $T$ has $\binom{p}{2}$ minimal forts corresponding to the possible pairs of pendant vertices $u_{1},\ldots,u_{p}$.
Therefore, applying Corollary~\ref{cor:minimal_fort_bound} gives us
\begin{align*}
\abs{\mathcal{F}_{T}} &\geq \abs{\mathcal{F}_{C_{1}}} + 1 + \sum_{j=2}^{q}\abs{\mathcal{F}_{C_{j}}} + \binom{p}{2} \\
&\geq \sum_{j=1}^{q}\frac{\abs{V(C_{j})}}{3} + \binom{p}{2} + 1 \\
&= \frac{n-p-1}{3} + \binom{p}{2} + 1 \geq \frac{n}{3} + 1,
\end{align*}
for $p\geq 2$.
\end{proof}

The next lemma shows that the number of star centers forms a lower bound on the fort number of trees where the star centers induce a connected graph.
\begin{lemma}\label{lem:fort_number_bound}
Let $T$ be a tree and let $\mathcal{S}_{T}$ denote the star centers of $T$. 
If $T[\mathcal{S}_{T}]$ is connected, then the fort number of $T$ satisfies 
\begin{equation}\label{eq:fort_number_bound}
\ft(T) \geq \abs{\mathcal{S}_{T}}. 
\end{equation}
\end{lemma}
\begin{proof}
We first show that if $T[\mathcal{S}_{T}]$ is connected, then $S_{1}=\emptyset$, so all star centers appear in $S_{0}$. 
For the sake of contradiction, suppose that there is a vertex $v\in S_{1}$.
Note that $v$ has fewer than two pendant neighbors in $T$.
Moreover, there exists a vertex $u\in S_{0}$ whose removal from $T$ causes $v$ to have at least two pendant neighbors in $T_{1}$. 
So, there exists a $w\in V(T)$ that is not a star center of $T$ and is adjacent to both $u$ and $v$ in $T$. 
Since $T$ is a tree, $u$ and $v$ are connected by a unique path.
Since this path contains a vertex $w$ that is not a star center, it follows that $T[\mathcal{S}_{T}]$ is not connected. 

Therefore, all star centers of $T$ are in $S_{0}$. 
So, each star center has at least two pendant neighbors in $T$.
Moreover, these pendant neighbors are distinct, that is, if $u,v\in S_{0}$ are distinct, then the pendant neighbors of $u$ and $v$ in $T$ are distinct. 
Since each pair of twin pendant vertices constitutes a minimal fort of $T$, it follows that 
\[
\ft(T) \geq \abs{S_{0}} = \abs{\mathcal{S}_{T}}. 
\]
\end{proof}

The next lemma characterizes the trees $T$ where $\ft(T)=\zf(T)=n/3$. 
\begin{lemma}\label{lem:ft_zf_equal}
Let $T$ be a tree of order $n=3k$, where $k\in\mathbb{N}$. 
If $\ft(T)=\zf(T)=k$, then $T=H\circ E_{2}$, where $H$ is a tree of order $k$. 
\end{lemma}
\begin{proof}
Suppose that $\ft(T) = \zf(T) = k$. 
Then, Theorem 4.15 of~\cite{Cameron2023} implies that the number of leaves satisfy $\ell(T)=2k$. 
Since $T$ is a tree, the zero forcing number is equal to the path cover number~\cite{AIM2008}. 
Therefore, $T$ has a minimum path cover $\mathcal{P}=\{P_{1},\ldots,P_{k}\}$. 
Since $\ell(T)=2k$, each path in $\mathcal{P}$ contains exactly two leaves. 
Since no path in $\mathcal{P}$ can be disconnected from $T$, every path in $\mathcal{P}$ has order at least $3$.
Since $T$ has order $n=3k$, it follows that every path in $\mathcal{P}$ has order exactly $3$.
Therefore, for each $i \in \{1, 2, \ldots, k\}$, we write the vertex set of $P_i$ as
\[
V(P_{i}) = \{x_{i}, y_{i}, z_{i}\}
\]
and the edge set of $P_{i}$ as
\[
E(P_{i}) = \{x_{i}y_{i}, y_{i}z_{i}\}, 
\]
where the vertices $x_{i}, z_{i}$ are pendant vertices in $T$. 
Since $T$ is connected, each $y_{i}$ is adjacent to a $y_{j}$ where $i \neq j$. 
We define $H$ to be the tree with vertex set 
\[
V(H) = \{y_i \mid 1 \leq i \leq k\}
\]
and edge set
\[
E(H) = \{y_{i}y_{j} \mid y_{i}y_{j}\in E(T)\}.
\]
Then $T = H \circ E_{2}$.
\end{proof}

We are now ready to prove a four-part characterization of the trees that attain the lower bound in Corollary~\ref{cor:minimal_fort_bound}. 
\begin{theorem}\label{thm:equiv_thm}
Let $T$ be a tree of order $n=3k$, where $k\in\mathbb{N}$.
Then, the following are equivalent:
\begin{enumerate}[(a)]
\item $\abs{\mathcal{F}_{T}}=k$. 
\item $\abs{\mathcal{S}_{T}}=k$ and $T[\mathcal{S}_{T}]$ is connected. 
\item $\ft(T)=\zf(T)=k$.
\item $T=H\circ E_{2}$, where $H$ is a tree of order $k$. 
\end{enumerate}
\end{theorem}
\begin{proof}
Suppose that (a) holds, that is, $\abs{\mathcal{F}_{T}}=k$.
Then, Theorem~\ref{thm:fort_star_bound} and Lemma~\ref{lem:star_center_bound} imply that $\abs{\mathcal{S}_{T}}=k$.
Moreover, Lemma~\ref{lem:minimal_fort_bound+1} implies that $T[\mathcal{S}_{T}]$ is connected; hence, (b) holds. 

Suppose that (b) holds, that is, $\abs{\mathcal{S}_{T}}=k$ and $T[\mathcal{S}_{T}]$ is connected. 
Then, by Lemma~\ref{lem:fort_number_bound}, $\ft(T)\geq k$. 
Since $T$ is connected, every fort contains at least $2$ vertices.
In fact, since there are $k$ fort-irrelevant vertices and $n=3k$, it follows that there are exactly $k$ disjoint forts, each of which contains exactly $2$ vertices.
Furthermore, a minimum zero forcing set can be constructed by selecting exactly one vertex from each of the $k$ disjoint forts. 
Therefore, $\ft(T)=\zf(T)=k$; hence, (c) holds. 

Suppose that (c) holds, that is, $\ft(T)=\zf(T)=k$. 
Then, Lemma~\ref{lem:ft_zf_equal} implies that $T=H\circ E_{2}$, where $H$ is a tree of order $k$. 

Suppose that (d) holds, that is, $T=H\circ E_{2}$, where $H$ is a tree of order $k$. 
Then, every vertex in $H$ is a star center of $T$.
Moreover, every vertex in $H$ has two pendant neighbors that form a minimal fort of $T$.
Therefore, $T$ has at least $k$ minimal forts. 
For the sake of contradiction, suppose there exists a minimal fort $F$ of $T$ that contains pendant neighbors from $2$ or more vertices in $H$.
Then, there exists $x,y\notin F$ such that $F$ has vertices in both components of $T-xy$, which contradicts condition (c) of Theorem~\ref{thm:mf_comb_cut_char}.
Therefore, $\abs{\mathcal{F}_{T}}=k$ and (a) holds. 
\end{proof}
\section{Conclusion}\label{sec:conclusion}
In this article, we investigated the structure and enumeration of minimal forts in trees. 
Our first main result provided a combinatorial-cut characterization of the minimal forts of a tree. 
This characterization shows that the structure of a minimal fort is tightly constrained by the local adjacency relationships among its vertices.

Using this characterization, we derived bounds on the cardinality and number of minimal forts in trees.
First, we demonstrated that paths (of order at least $6$) and spiders have at least $n/2$ minimal forts. 
We also showed that the $n/2$ bound holds for all trees with no star centers. 
We then examined the relationship between minimal forts and star centers. 
In particular, we established an inequality relating the number of minimal forts in a tree to the number of star centers. 
This relationship leads to a general lower bound on the number of minimal forts in trees;  in particular, every tree of order $n$ has at least $n/3$ minimal forts. 
Finally, we characterized the trees that attain this lower bound. 
In particular, we proved a four-part equivalence theorem involving trees that attain this lower bound and their star centers, fort number, and zero forcing number. 

Several directions remain for future work. 
While the characterization obtained in this paper relies on the acyclic structure of trees, we conjecture that the $n/3$ lower bound on the number of minimal forts holds for all graphs.
It would be interesting to determine how these structural arguments and the relationship between forts and star centers extend to more general graph classes.
In addition, we are interested in studying the trees with a maximum number of minimal forts.
\section*{Acknowledgments}
The authors would like to acknowledge Dr. Joe Previte (Penn State Behrend) for first suggesting the $n/3$ lower bound on the number of minimal forts.
In addition, the authors would like to thank Dr. Boris Brimkov (Slippery Rock University) and Dr. Houston Schuerger (Tarleton State University) for many helpful conversations that supported this work.
\bibliographystyle{siam}
\bibliography{Bibliography}
\end{document}